\documentclass{aims_hyp_2}
\usepackage{amsmath}
  \usepackage{paralist}
  \usepackage{graphics} 
  \usepackage{epsfig} 
 \usepackage[colorlinks=true]{hyperref}
\hypersetup{urlcolor=blue, citecolor=red}

\def\currentvolume{X}
 \def\currentissue{X}
  \def\currentyear{200X}
   \def\currentmonth{XX}
    \def\ppages{X--XX}

\newtheorem{theorem}{Theorem}[section]
\newtheorem{corollary}{Corollary}

\newtheorem{lemma}[theorem]{Lemma}
\newtheorem{proposition}{Proposition}

\theoremstyle{definition}
\newtheorem{definition}[theorem]{Definition}
\newtheorem{remark}{Remark}

\title[Robin boundary condition]
      {Degenerate Convection-Diffusion  Equation with a Robin boundary condition}

\author[Mohamed-Gazibo Karimou]{}

 \email{mgazibok@univ-fcomte.fr}
\thanks{The author is supported by ANR CoToCoLa (Contemporary Topics on Conservation Laws).}

\begin{document}
\maketitle

\centerline{\scshape  Mohamed Gazibo Karimou  }
\medskip
{\footnotesize
 \centerline{Laboratoire de Math\'{e}matiques, CNRS : UMR 6623}
   \centerline{Universit\'{e} de Franche-Comte, 16, route de Gray, 25030 Besan\c con-France}
   } 

\bigskip

 \centerline{(Communicated by the associate editor name)}

\maketitle
\begin{abstract}
We study a Robin boundary problem for degenerate parabolic equation. We suggest a notion of entropy solution and propose a result of existence and uniqueness. Numerical simulations illustrate some aspects of solution behavior.
\end{abstract}

\section{Introduction}
Let $\Omega$ be an open bounded domain  of $\mathbb R^\ell$ with a Lipschitz boundary $\partial\Omega$, and $\eta$ the unit normal to $\partial\Omega$ outward to $\Omega$.
The purpose of this paper is to discuss existence and uniqueness of entropy solution for the following initial boundary value problem
\begin{equation*}
(P)\left\{\begin{array}{lll}
u_t+\mathop{\rm div} f(u)-\Delta\phi(u)&=0\;\;\;\;\;\;\;\;\;\;\;\;\;\;\;\;\;\;\;\mbox{ in } \;\;\; Q=]0,T[\times\Omega,\\
u(0,x)&=u_0(x)\;\;\;\;\;\;\;\;\;\;\;\;\mbox{ in } \;\;\; \Omega,\\
b(u)-(f(u)-\nabla\phi(u)).\eta&=0\;\;\;\;\;\;\;\;\;\;\;\;\;\;\;\;\;\;\mbox{ on } \;\;\; \Sigma=]0,T[\times\partial\Omega.
\end{array}\right.
\end{equation*}
Here, $u_0$ is taking values on $[0, u_{\max}]$ for some $u_{\max}>0$. Further, the function $f$ is a Lipschitz continuous function. Moreover, we require that
\begin{equation}\label{f}\tag{H1}
f(0)=0\;\mbox{ and } \;\;b(0)=0.
\end{equation}
The diffusion term $\phi$ is a continuous function. We consider that there exist a critical value $u_c$ of the unknown $u$ such that:
$\phi(.)$ is zero  on $[0,u_c]$ with $0\leq u_c\leq u_{\max}$ and $\phi(.)$ is strictly increasing else. Then problem $(P)$ degenerates to hyperbolic when $u$ takes values in the region $[0,u_c]$ where $\phi$ is flat.\\ We suppose that the function $b$ is a continuous non-decreasing function on $\Sigma$. In some situation, $b$ may be a maximal monotone graph on $\mathbb R$ (see \cite{BS}). Here, we assume also that $b$ satisfies the following hypotheses: 	
 \begin{equation*}\label{btrace}\tag{H2}
b=\beta\circ\phi\mbox{ where $\beta$ is a non-decreasing  Lipschitz continuous function. }
\end{equation*}
\begin{equation*}\label{bmax}\tag{H3}
b(u_{\max})\geq|f(u_{\max}).\eta|.
\end{equation*}
For more than a few decades, the degenerate parabolic equation in bounded domain was studied by many authors mainly in the case of Dirichlet  boundary conditions (see e.g. \cite{PORETTA}, \cite{CAR}). The zero-flux boundary condition is studied in \cite{BF} for non-degenerate parabolic case, in \cite{BFK1} for fully degenerate hyperbolic equation and recently in \cite{GB} for the parabolic-hyperbolic problem. Remark, that the condition $b(u)-(f(u)-\nabla\phi(u)).\eta=0$ on $\Sigma$ includes in particular Neumann (zero-flux) condition on the boundary. 

We propose an adequate entropy formulation for problem  $(P)$ which incorporates two boundary integrals. In \cite {GB}, existence and uniqueness for the zero flux boundary condition were proved, under the assumption \eqref{bmax} that reads $f(u_{\max})=0$ in the zero-flux case $b\equiv0$. In contrast to the entropy formulation in \cite{GB}, where the passage to the limit in the only boundary integral is straightforward, for our entropy inequality, we need the assumption \eqref{btrace}, which permits to give a sense to the boundary integral with the term  $b(u)$. Indeed, we can deduce that $b(u)$ has a trace on the boundary as a function in Sobolev space $H^1(\Omega)$.

The proof of existence of our entropy solution for any space dimensions $\ell\geq1$ employs a vanishing viscosity approximation. We pass to limit in the interior of the domain $Q$, by using the local compactness result of Panov \cite {PAN}, for this we suppose some relation between $f$ and $\phi$ (see Definition \ref{compacite}). One can refer to  \cite{GB} for more details. We pay particular attention to the boundary term  (here \eqref{btrace} is needed).

For the uniqueness result, we use nonlinear semigroup techniques (see, e.g., \cite{BGP}) and Kruzhkov doubling of variables methods.  The main goal is to compare two solutions of (P), and it turns out that it is simpler to compare a solution of $(P)$ with a regular solution (in the sense that the total flux is continuous up to the boundary) of the stationary problem associated to $(P)$. Then we prove that entropy solution of $(P)$ is an integral solution, and we refer to the uniqueness of integral solutions granted by the general theory of nonlinear semigroup. Unfortunately, we are not able to obtain regular solution to the stationary problem for any space dimensions, but only in one space dimension. Then, we can deduce the uniqueness just now when $\Omega$ is a bounded open interval of $\mathbb R$. Notice that, for the same argument as for the zero-flux boundary condition \cite{GB}, the problem of uniqueness is still open in multiple space dimensions.

The paper is organized as follows. In  the next section, we give our definition of entropy solution and state some remarks useful for the well-possedness. In section 3, we  prove existence result of entropy  solution.  In the section 4, we prove uniqueness in the case of one space dimension. The latter part is devoted to the numerical investigation of problem $(P)$. We adapt the approach of finite volumes in the spirit of Vovelle (\cite{VOV}) to illustrate  and interpret some observations in the case where the assumptions \eqref{btrace} and \eqref{bmax} are absent. Thereby, we justify the importance these assumptions in this paper.

\section{Notion of entropy solution} Consider the following notion.
\begin{definition}\label{entrsol}
A measurable function $u$ taking values on $[0,u_{\max}]$ is called entropy solution of problem $(P)$ if $\phi(u)\in L^2(0,T;H^1(\Omega))$, $b(u)\in  L^2(0,T;H^1(\Omega))$ and the following conditions hold: \\$\forall k\in  [0,u_{\max}]$,  $\forall\xi\in \mathcal{C}_0^\infty([0,T[\times\mathbb R^\ell)$, with $\xi\geq 0$:
\begin{align}\label{ESP}
&\displaystyle\int_0^T\int_\Omega\left\{|u-k|\xi_t+sign(u-k)\biggl(f(u)-f(k)-\nabla\phi(u)\biggr).\nabla\xi\right\}dxdt\nonumber\\&+\displaystyle\int_\Omega |u_0-k|\xi(0,x)dx+\displaystyle\int_0^T\int_{\partial\Omega} \left|f(k).\eta(x)-b(k)\right|\xi(t,x) d{\mathcal{H}}^{\ell-1}dt
\nonumber\\&-\displaystyle\int_0^T\int_{\partial\Omega}|b(u)-b(k)|\xi(t,x) d{\mathcal{H}}^{\ell-1}dt\geq 0.
\end{align}
Here $\mathcal{H}$ represents the $(\ell-1)-$ dimensional Hausdorff measure on $\partial\Omega$.
\end{definition}
\begin{remark}\label{entweak}
\begin{enumerate}
\item The entropy solution in the sense of Definition \ref{entrsol}  is in particular a weak solution. Indeed, first take in inequality \eqref{ESP}, $k=0$ and use \eqref{f}. Next, take $k=u_{\max}$ and use \eqref{bmax}.
\item Let us stress that, in particular, the  boundary condition $(f(u)-\nabla\phi(u)).\eta=b(u)$ is verified literally in the weak sense as in the case of zero flux boundary condition (see \cite{GB}). This contrasts with the properties of the Dirichlet problem (see \cite{BNL}); we expect that the  boundary condition should be relaxed if assumption \eqref{bmax} is dropped (see \cite{BS, BSHIBI} and also numerical tests of section 5).
\item The integral in the boundary term is well defined due to the hypothesis \eqref{btrace}. We can use the fact that the trace of $b(u)(t,.)\in H^1(\Omega)$ on $\partial\Omega$ is well defined in $L^2(\partial\Omega)$ for $t\in (0,T)$ a.e.
\end{enumerate}
\end{remark}
According to the idea of J. Carrillo (cf \cite{CAR}), we give an additional property of entropy solutions, useful for the  uniqueness techniques.
\begin{proposition}\label{carrillo}
 Let $\xi\in \mathcal{C}^\infty([0,T[\times\mathbb R^\ell)$; then for all $k\in[\phi_{c},u_{\max}]$; for all $D\in\mathbb R^\ell$ and for all entropy solution $u$ of $(P)$, we have:
\begin{align}\label{procar}
&\displaystyle\int_0^T\int_\Omega\left\{|u-k|\xi_t+sign(u-k)(f(u)-f(k)-\nabla\phi(u)+D).\nabla\xi\right\}dxdt\nonumber\\&+\displaystyle\int_\Omega|u_0-k|\xi(0,x)dx-\displaystyle\int_0^T\int_{\partial\Omega} \left|b(u)-b(k)\right|\xi(t,x) d{\mathcal{H}}^{\ell-1}dt\nonumber\\&+\displaystyle\int_0^T\int_{\partial\Omega} \left|b(k)-(f(k)-D).\eta(x)\right|\xi(t,x) d{\mathcal{H}}^{\ell-1}dt\nonumber\\&\geq\mathop{\overline{\lim}}\limits_{\sigma\rightarrow0}\frac{1}{\sigma}\displaystyle\int_0^T\int_{Q\cap\left\{-\sigma<\phi(u)-\phi(k)<\sigma\right\}}\nabla\phi(u).(\nabla\phi(u)-D)\xi(t,x)dxdt.
\end{align}
\end{proposition}
In general, uniqueness for evolution equation of kind $(P)$ appear very difficult mainly for the initial boundary values problems. In this context, the use of nonlinear semigroup techniques offers many advantages. Let us present briefly another notion of solution coming from the theory of nonlinear semigroups (see, e.g., \cite{BGP}).
\begin{definition}
 Let $A$ be an m-accretive operator (see, e.g., \cite{BGP}). Suppose that $h\in L^1(Q)$,  $u_0\in L^1(\Omega)$. A measurable function $v\in \mathcal{C}([0,T]; L^1(\Omega;[0,u_{\max}]))$\footnote{Here, we will write  $L^1(\Omega;[0,u_{\max}])$ for the set of all measurable functions from $\Omega$ to $[0,u_{\max}]$.} is an integral solution of the abstract evolution problem
\begin{align}\label{solintegrale}
v_t + A(v)\ni h(t),\;\;\;\; v(t = 0) = u_0,
\end{align}
if $v(0,.) = u_0(.)$ and for all $(u,z) \in A$
$$\frac{d}{dt}||v(t)-u||_{L^1(\Omega)}\leq\int_\Omega sign_0(v(t)-u)(h(t)-z)+\int_{\{v=u\}}|h(t)-z| \mbox{ in } \mathcal{D}'(0,T).$$
\end{definition}
We will see that entropy and integral solution coincide in the case $\Omega=(a,b)$ an interval of $\mathbb R.$
\section{Existence of entropy solution}
The main result of this part is the following:
\begin{theorem}\label{existence1}
Let $\ell\geq 1$. Assume that \eqref{f}, \eqref{btrace} and \eqref{bmax} holds. Suppose that $(f,\phi)$ is non-degenerate (in the sense of Definition \ref{compacite} below). Then, there exists an entropy solution $u$ for the problem $(P)$.
\end{theorem}
To show the existence of entropy solutions,  we approximate $\phi(u)$ by $\phi_\epsilon(u^\epsilon)=\phi(u^\epsilon)+\epsilon Id(u^\epsilon)$ for each $\epsilon>0$ and set $b_\epsilon(u^\epsilon)=\beta\circ\phi_\epsilon(u_\epsilon)$.
 We obtain the following regularized strictly parabolic problem $(P_\epsilon)$ with unknown $u^\epsilon$
\begin{equation*}
(P_\epsilon)\left\{\begin{array}{lll}
u_t^\epsilon+\mathop{\rm div}  f(u^\epsilon)-\Delta\phi_\epsilon(u^\epsilon)&=0\;\;\;\;\;\;\;\;\;\;\;\;\;\;\;\;\;\;\;\mbox{ in } \;\;\; Q=]0,T[\times\Omega,\\
u^\epsilon(0,x)&=u_0^\epsilon(x)\;\;\;\;\;\;\;\;\;\;\;\mbox{ in } \;\;\; \Omega,\\
b_\epsilon(u^\epsilon)-(f(u^\epsilon)-\nabla\phi_\epsilon(u^\epsilon)).\eta&=0\;\;\;\;\;\;\;\;\;\;\;\;\;\;\;\;\;\;\;\mbox{ on } \;\;\; \Sigma=]0,T[\times\partial\Omega,
\end{array}\right.
\end{equation*}
where $(u_0^\epsilon)_{\epsilon}$ is a sequence of smooth functions that converges to $u_0$ a.e and respects the minimum/maximum values of $u_0$.
\begin{definition}\label{weaksolappro}
Let $u_0$ be a measurable $[0,u_{\max}]$-valued function. A measurable function $u^\epsilon\in L^2(0,T;H^1(\Omega))$ taking values on $[0,u_{\max}]$ is called  weak solution of problem $(P_\epsilon)$ if :  $\forall \theta\in L^2(0,T;H^1(\Omega))\cap L^\infty(Q)$ such that $\theta_t\in L^2(Q)$ and $\theta(T,.)=0$, one has
\begin{align}\label{WSP}
&\displaystyle\int_0^T\int_\Omega\left\{u^\epsilon\theta_t+(f(u^\epsilon)-\nabla\phi_\epsilon(u^\epsilon)).\nabla\theta\right\}dxdt+\int_\Omega u_0^\epsilon\theta(0,x)dx\nonumber\\&
-\int_0^T\int_{\partial\Omega}b_\epsilon(u^\epsilon)\theta d{\mathcal{H}}^{\ell-1}dt=0.
\end{align}
\end{definition}
\begin{theorem}\label{existence2}
For $u_0\in [0,u_{\max}]$, assume \eqref{f}, \eqref{btrace} and \eqref{bmax} hold. Problem $(P_\epsilon)$ admits a weak solution $u^\epsilon$ which is also an entropy solution. In particular, we have $0\!\leq u^\epsilon\!\leq u_{\max}$. In addition, there exists  C independent on  $\epsilon$ such that
\begin{align}\label{estimateur1}
||\sqrt{\epsilon}\nabla u^\epsilon||_{L^2(Q)}\leq C;
\end{align}
\begin{align}\label{estimateur2}
||\phi_\epsilon(u_\epsilon)||_{L^2(0,T; H^1(\Omega))}\leq C;
\end{align}
\begin{align}\label{estimateur3}
||b_\epsilon^n(u_\epsilon)||_{L^1(\Sigma)}\leq C \mbox{ and } \displaystyle\int_\Sigma u_\epsilon b_\epsilon^n(u_\epsilon)\leq C.
\end{align}
\end{theorem}
This result can be proved, e.g., using Galerkin method (see e.g. \cite{GB}).
\begin{lemma}\label{boundary}
Assume that the sequence $(\Psi_j)_j$ is such that: $||\Psi_j||_{L^2(0,T;H^1(\Omega))}\leq C$ and $\Psi_j\longrightarrow\Psi$ in $L^2(Q)$. Then $\gamma\Psi_j\longrightarrow\gamma\Psi$ in $L^2(\Sigma)$, where $\gamma$ is the trace operator.
\end{lemma}
The proof uses localization to a small neighborhood of $\Sigma$.\\
To prove existence of entropy solution,  we assume that the couple $(f(.),\phi(.))$ is non-degenerate in the sense of the following definition:
\begin{definition}\label{compacite}(\normalfont{{Panov} \cite{PAN}}).
Let $\phi$ be zero on $[0,u_c]$, strictly increasing on $[u_c,u_{\max}]$ and a vector $f=(f_1,...,f_\ell)$. A couple $(f(.),\phi(.))$ is said to be non-degenerate if, for all $\xi\in\mathbb R^{\ell}\backslash \{0\}$, the functions $\lambda\longmapsto\displaystyle\sum\nolimits_{i=1}^{\ell}\xi_if_i(\lambda)$ are not affine on the non-degenerate sub intervals of $[0, u_c]$.
\end{definition}
\begin{theorem}\label{panov}(\normalfont{{Panov} \cite{PAN}}).
Assume that $(f,\phi)$ is non degenerate in the sense of Definition \ref{compacite}. Suppose $u^\epsilon$, $\epsilon>0$, is a sequence such that
\begin{align*}
&\exists d>1, \forall s,r\in\mathbb R\mbox{  with } s<r\\
&T_{s,r}(u^\epsilon)_t+\mathop{\rm div} \Bigl(f(T_{s,r}(u^\epsilon))-\nabla\phi(T_{s,r}(u^\epsilon))\Bigr)\mbox{ is pre-compact in } W_{\mbox{{\tiny Loc}}}^{-1,d}(Q).
\end{align*}
 Moreover, suppose $u^\epsilon$, $f(u^\epsilon)$, $\phi_\epsilon(u^\epsilon)$ are equi-integrable locally on $Q$. Then, there exists subsequence
$(u^\epsilon)_{\epsilon}$ that converges in $L^1_{\mbox{{\tiny Loc}}}(Q)$.
\end{theorem}
\begin{proof}[Proof of Theorem~\ref{existence1}](Sketched)
The proof of existence of entropy solution uses Theorem \ref{panov} to justify the passage to the limit in $Q$ (for more details, see \cite{GB}) and Lemma \ref{boundary} for boundary integral.
\end{proof}
\section{Uniqueness result of entropy solution in one space dimension}
The main result  of this section is the following theorem:
\begin{theorem}\label{unicite}
Suppose that $\Omega=(a,b)$ is a bounded interval  of $\mathbb R$, then $(P)$ admits a unique entropy solution.
\end{theorem}
In order to study uniqueness in the framework of nonlinear semigroup theory, we consider for all bounded function $g$ taking values on $[0,u_{\max}]$,  the stationary problem $(S)$ associated to  problem $(P)$:
\begin{equation*}
(S)\left \{\begin{array}{rl}
  u+\mathop{\rm div} (f(u)-\nabla\phi(u))&=g\;\mbox{ in } \; \Omega,\\
  b(u)-(f(u)-\nabla\phi(u)).\eta&=\;\mbox{ on } \; \partial\Omega.
\end{array} \right.
\end{equation*}
The notion of entropy solution of $(S)$ correspond to the time-independent entropy solution of $(P)$ with source term $g-u$. In the case where $\Omega=(a,b)$ is a bounded interval  of $\mathbb R$, we have an important result, which states that, the total flux is regular at the points $a$ and $b$. This kind of regularity seem hard to obtain in multiple space dimensions for $(S)$, and even in dimension $\ell=1$ for $(P)$.
\begin{proposition}\label{traceforte}
For all measurable function $g$ taking values in $[0,u_{\max}]$ the problem $(S)$ admits a solution $u$ such that $(f(u)-\phi(u)_y)$  is continuous up the boundary, i.e., $(f(u)-\phi(u)_y)\in \mathcal{C}([a,b])$. Moreover, $b(u)-(f(u)-\phi(u)_y).\eta(y)$ is zero at $y=a$ and $y=b$. (Here $\eta(a)=-1$ and $\eta(b)=+1$).
\end{proposition}
From now, let's define the operator $A_{f,\phi,b}$ on $L^1$ associated with  regular solutions of $(S)$ by its graph:
\begin{equation*}
(u,z)\!\in A_{f,\phi,b}=\!\left \{\begin{array}{ll}
 \!\!u \mbox{ such that } u \mbox{ is an entropy solution of } (S), \mbox{ with } g=u+z\!\!\!
\end{array} \right\}\!\!.
\end{equation*}
 \begin{proposition}\label{maccretive}
\begin{enumerate}
\item $A_{f,\phi,b}$ is accretive in $L^1(\Omega)$.
\item For all $\lambda$ sufficiently small, $R(I\!+\lambda A_{f,\phi,b})$ contains $L^1(\Omega;[0,u_{\max}])$.
\item $\overline{D(A_{f,\phi,b})} = L^1(\Omega;[0,u_{\max}])$.
\end{enumerate}
\end{proposition}
For the proof of this proposition, we can refer to \cite{GB}.\\
According to the general results of  \cite{BGP}, it follows existence and uniqueness of integral solution in the sense of Definition \ref{entrsol}:
\begin{corollary}\label{unik1}
Let $\Omega=(a,b)$, $u_0,\hat{u}_0\in L^1(\Omega)$ and $h,\hat{h}\in L^1(Q)$. Let $v,\hat{v}$ be integral solutions of \eqref{solintegrale} (with operator $A_{f,\phi,b}$ ) associated with the data $(u_0,h)$ and $(\hat{u}_0,\hat{h})$, respectively. Then for a.e. $t\in[0,T)$.
\begin{equation*}
||v(t)-\hat{v}(t)||_{L^1}\leq||u_0-\hat{u}_0||_{L^1}+\displaystyle\int_0^t ||h(\tau)-\hat{h}(\tau)||_{L^1}dt.
\end{equation*}
\end{corollary}
Adapted to our case, we have the following result
\begin{theorem}\label{unik}
Let $\Omega=(a,b)$. Let $v$ be an entropy solution of $(P)$ and $u$ be an entropy solution of $(S)$. Then
\begin{equation}\label{insol}
\frac{d}{dt}||v(t)-u||_{L^1(\Omega)}\leq\displaystyle\int_\Omega sign(v-u)(u-g)dx  \mbox{ in  } \mathcal{D}'(0,T).
\end{equation}
In particular, $v$ is  an integral solution of \eqref{solintegrale} with $h=0$.
\end{theorem}
\begin{proof}[Proof of Theorem~\ref{unik} and Theorem~\ref{unicite}]  We consider $v=v(t,x)$ an entropy solution of $(P)$ and $u=u(y)$  an entropy solution of $(S)$. Consider nonnegative function $\xi=\xi(t,x,y)$ having the property that $\xi(.,.,y)\in \mathcal{C}^\infty([0,T)\times\overline{\Omega})$ for each $y\in\overline{\Omega}$, $\xi(t,x,.)\in \mathcal{C}_0^\infty(\overline{\Omega})$ for each $(t,x)\in[0,T)\times\overline{\Omega}$. Apply the doubling of variables \cite{KRU} in the spirit of \cite{GB}, we obtain this following inequality
\begin{align}\label{u7}
&\displaystyle\int_0^T\!\!\int_\Omega\!\int_\Omega |v-u|\xi_tdydxdt+\displaystyle\int_\Omega\int_\Omega|v_0-u|\xi(0,x,y)dxdy\nonumber\\&+\int_0^T\!\!\!\int_\Omega\!\int_\Omega sign(v-u)\Bigl[(f(v)-\phi(v)_x)-(f(u)+\phi(u)_y)\Bigr].(\xi_x+\xi_y)dxdydt\nonumber\\&+\displaystyle\int_0^T\int_{x\in\partial\Omega}\!\int_\Omega \left|b(u)-(f(u)-\phi(u)_y).\eta(x)\right|\xi dyd\sigma dt\nonumber\\&+\int_0^T\!\!\!\int_\Omega\!\int_{y\in\partial\Omega}\left|b(v)-(f(v)-\phi(v)_x).\eta(y)\right|\xi d\sigma dxdt\nonumber\\&
+\int_0^T\!\!\int_\Omega\!\int_\Omega sign(v-u)(u-g(y))\xi dydxdt\nonumber\\&\geq\displaystyle\int_0^T\!\!\!\int_{x\in\Omega}\int_{y\in\partial\Omega}\left|b(u)-b(v)\right|\xi d\sigma dxdt+\displaystyle\int_0^T\!\!\!\int_{y\in\Omega}\int_{x\in\partial\Omega}\left|b(u)-b(v)\right|\xi d\sigma dydt\nonumber\\&+\mathop{\overline{\lim}}\limits_{\sigma\rightarrow0}\frac{1}{\sigma}\displaystyle\int_0^T\int\!\!\!\int_{{{\Omega_x^c}\times{\Omega_y^c}}\cap\left\{-\sigma<\phi(v)-\phi(u)<\sigma\right\}}|\phi(v)_x-\phi(u)_y|^2\xi dydxdt\geq0.
\end{align}
Next, following the idea of \cite{BF}, we take the test function $\xi(t,x,y)\!=\!\theta(t)\rho_n(x,y)$, where $\theta\!\in\!\mathcal{C}_0^\infty(0,T)$, $\theta\!\geq\!0$, $\rho_n(x,y)\!=\!\delta_n(\Delta)$ and  $\Delta\!=\!(1-\frac{1}{n(b-a)})x-y+\frac{a+b}{2n(b-a)}$. Then, $\rho_n\in \mathcal{D}(\overline\Omega\times\overline\Omega)$ and $\rho_{{n}_{|_{\Omega\times\partial\Omega}}}(x,y)=0$. Due to this choice, $$\displaystyle\int_0^T\!\!\!\int_{x\in\Omega}\int_{y\in\partial\Omega}\left|b(v)-(f(v)-\phi(v)_x).\eta(y)\right|\rho_n\theta dyd\sigma dt=0.$$
By Proposition \ref{traceforte}, $b(u)-(f(u)-\phi(u)_y).\eta(y)\in \mathcal{C}_0([a,b])$. Therefore we have\\ $\left|b(u)-(f(u)-\phi(u)_y).\eta(x)\right|\longrightarrow 0$ when $x\rightarrow y$, i.e, as $n\longrightarrow\infty$. We conclude that $$\lim_{n\rightarrow\infty}\displaystyle\int_0^T\!\!\!\int_{x\in\partial\Omega}\int_{y\in\Omega}\left|b(u)-(f(u)-\phi(u)_y).\eta(x)\right|\rho_n\theta dyd\sigma dt=0.$$
with the calculation detailed in \cite{GB}, we deduce that
\begin{equation*}
\displaystyle\int_0^T\!\!\int_\Omega\!\int_\Omega \theta sign(v-u)\Bigl[(f(v)-\phi(v)_x)-(f(u)-\phi(u)_y)\Bigr].\bigl((\rho_n)_x+(\rho_n)_y\bigr) dydxdt\rightarrow 0.
\end{equation*}
Hence, we get \eqref{insol} by passing to the limit in \eqref{u7} with the above choice of $\xi$.
Thus, the entropy solution $v$ of the problem $(P)$ is an  integral solution of \eqref{solintegrale}. This proves that $v$ is a unique entropy solution  due to Corollary \ref{unik1}.
\end{proof}
\section{Role of hypotheses (\ref{btrace}), (\ref{bmax})  and some numerical illustrations }
The numerical analysis of $(P)$ is not the aim of this paper, although we consider this  alternative in a future work. We assume \eqref{f} holds, $u_c=0.6$ and $u_{\max}=1$. We present briefly the importance of the hypotheses \eqref{btrace}, \eqref{bmax}. We apply  now the ideas developed e.g., in the work  of  Vovelle (\cite{VOV}) to construct a monotone finite volume scheme which take into account the boundary condition. The interval $[0,1]$ is divided into $I$ cells. We initialize the scheme by:
\begin{align}
\forall i\in\{1,...,I\}: u_i^0=\frac{1}{\delta x}\displaystyle\int_{(i-1)\delta x}^{i\delta x} u_0(x)dx,
\end{align}
the numerical approximation solution at $t=n\delta t$ in the cell number $i\in\{2,...,I-1\}$ is :
\begin{align}\label{est1}
 u_i^{n+1}\!=u_i^n\!-\frac{\delta t}{\delta x}\biggl(F(u_{i}^n,u_{i+1}^n)\!-F(u_{i-1}^n,u_{i}^n)\!-\frac{\phi(u_{i+1}^n)\!-2\phi(u_i^n)\!+\phi(u^n_{i-1})}{\delta x}\biggr)
\end{align}
with the boundary conditions taken into account via
\begin{align}\label{esti2}
 u_1^{n+1}\!=u_1^n\!-\frac{\delta t}{\delta x}\biggl(F(u_{1}^n,u_{2}^n)\!-\frac{\phi(u_{2}^n)\!-\phi(u_1^n)}{\delta x}-b(u^n_1)\biggr).
\end{align}
\begin{align}\label{esti3}
 u_I^{n+1}\!=u_I^n\!-\frac{\delta t}{\delta x}\biggl(b(u^n_I)-F(u_{I-1}^n,u_{I}^n)\!+\frac{\phi(u_{I}^n)\!-\phi(u_{I-1}^n)}{\delta x}\biggr).
\end{align}
Here, $F$ is a numerical flux which we assume monotone, consistent, Lipschitz continuous (see \cite{VOV}).  In the sequel, we take $u_0(x)=0.7$ if $x\in[\frac{1}{2},1]$ and $u_0(x)=0$ if $x\in[0,\frac{1}{2}[$. We take  $\delta x=0.01$, $\delta t=\frac{\delta x^2}{5}$, $\phi(u)=(u-0.6)^+$ and consider a numerical solution at time $t=0.12$. Initially, we remove the hypothesis \eqref{bmax}, by taking $f(u)=\frac{u^2}{2}$ and $b(u)=\phi(u)$, we observe numerically  the loss of maximum principle (see Figure \ref{fig1} ), this mean that the the solution u can be greater than $u_{\max}$.  Our entropy formulation requires to choose $b(u)$ in the functional space that permit to define the trace of $b(u)$ on the boundary. In the context where assumption \eqref{btrace} is not taken into account, $b(u)=u$ and $f(u)=u(1-u)1_{[0,1]}$;  numerically, we observe a boundary layer (see Figure \ref{fig2} ) and this is confirmed by theoretical results of \cite{BS}. Now, taking into account assumptions \eqref{bmax}, \eqref{btrace}, with data $f(u)=u(1-u)1_{[0,1]}$; $b(u)=\phi(u)$ the numerical observation shows that the boundary condition at $x=0$ and $x=1$ is verified literally and the numerical solution respect the maximum principle (see Figure \ref{fig3}).
    \begin{figure}[!ht]
       \begin{minipage}[b]{0.4\linewidth}
          \centering \includegraphics[width=2.2in]{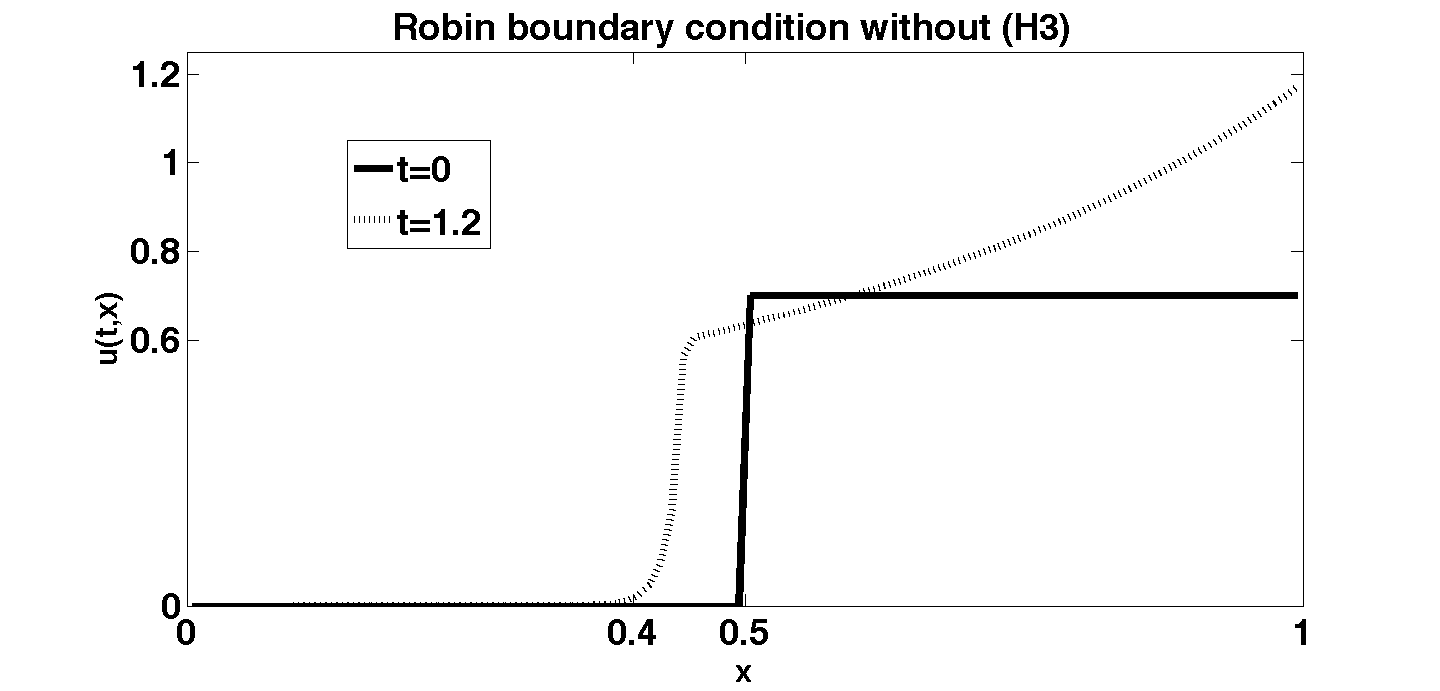}\caption{}\label{fig1}
       \end{minipage}\hfill
       \begin{minipage}[b]{0.4\linewidth}
      
       \centering \includegraphics[width=2.2in]{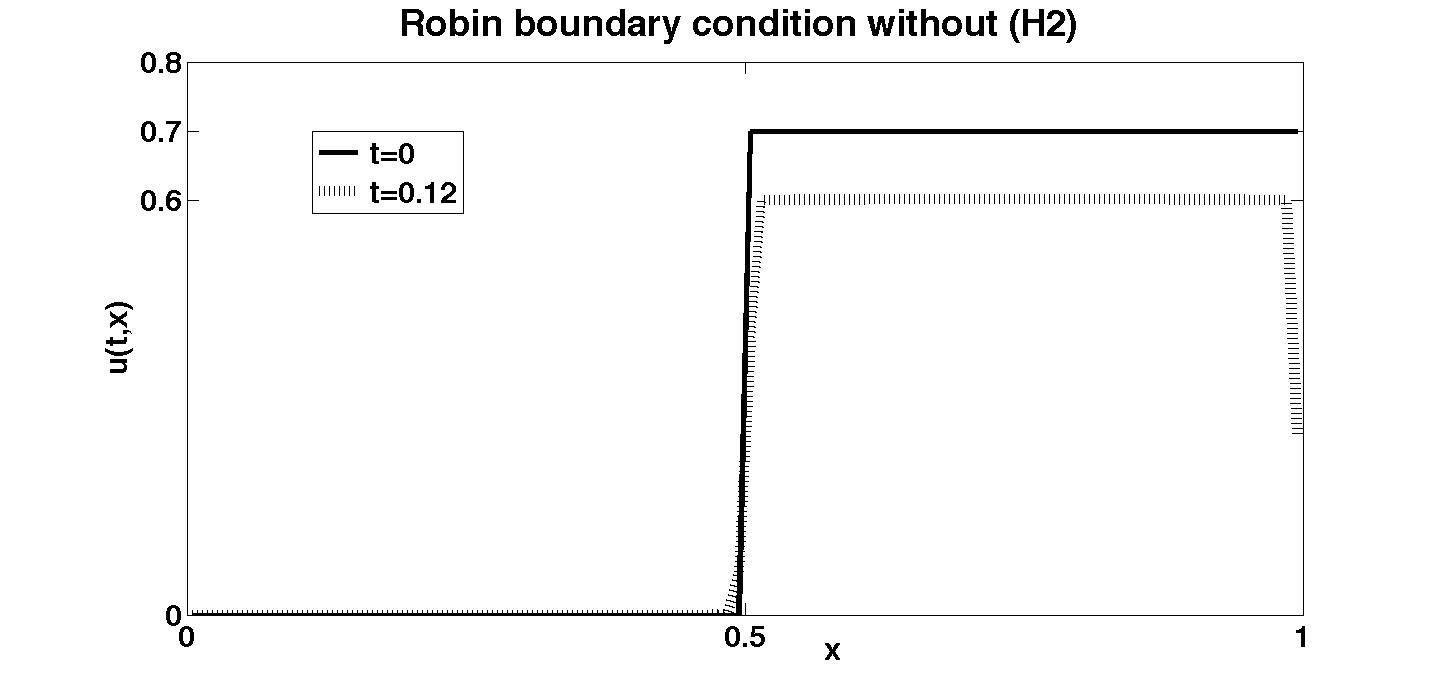}\caption{}\label{fig2}
       \end{minipage}\hfill
       \begin{minipage}[b]{0.4\linewidth}
          \centering \includegraphics[width=2.2in]{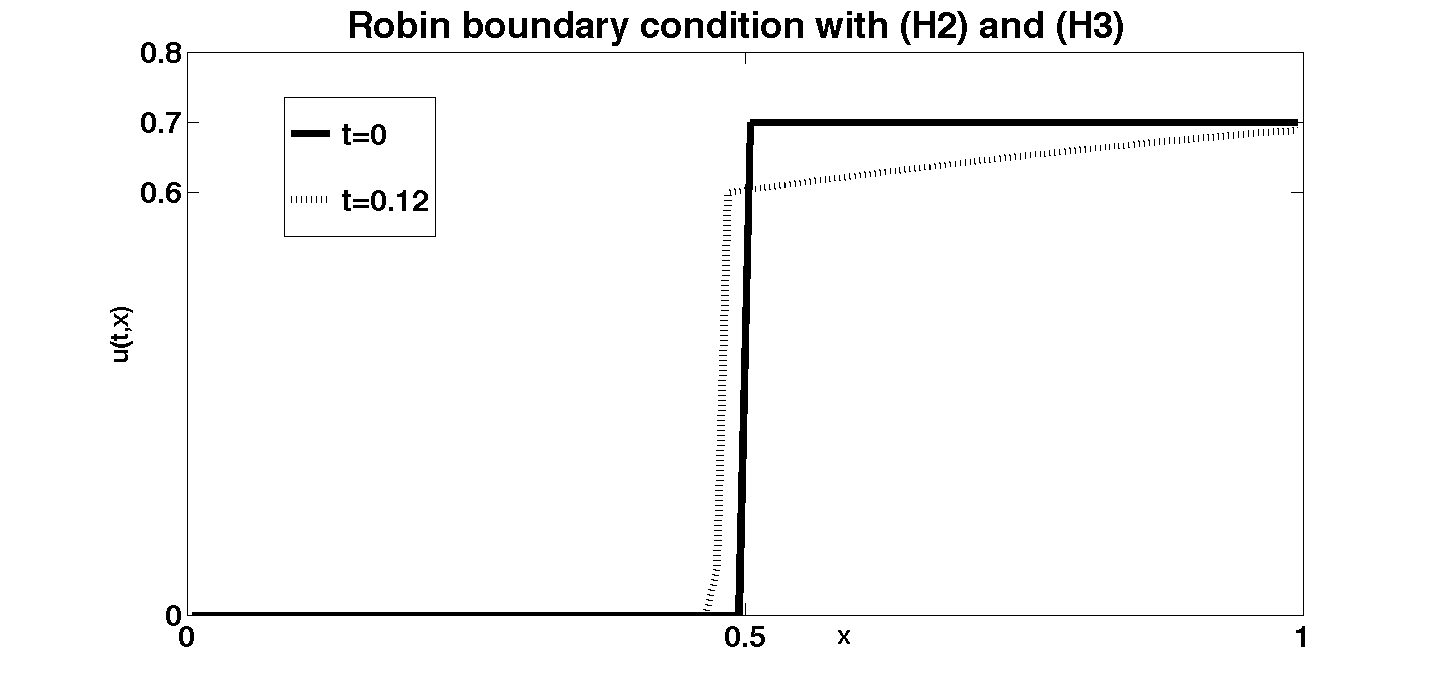}\caption{}\label{fig3}
       \end{minipage}\hfill
    \end{figure}
%
\section*{Acknowledgments} I would like to thank Boris Andreianov for his thorough reading and helpful remarks which helped me improve this paper.

\medskip
\medskip

\end{document}